\documentclass[11pt]{amsart}
\usepackage{amsmath}
\usepackage{graphicx, xcolor}
\newtheorem{theorem}{Theorem}[section]
\newtheorem{prop}{Proposition}[section]
\newtheorem{lemma}{Lemma}[section]
\newtheorem{corollary}{Corollary}[section]
\newtheorem{definition}{Definition}[section]
\newtheorem{notation}{Notation}[section]
\newtheorem{remark}{Remark}[section]

\newcommand{\E}{\mathbb{E}} 
 
\newcommand{\calP}{\mathcal{P}}

\numberwithin{equation}{section}

\def\C{\mathcal{C}}
\def\R{[0,\infty)}
\def\Reals{\mathbb{R}}
\def\intR{\int_\Reals}

\newcommand{\N}{\mathbb{N}}

\newcommand{\bp}{\begin{proof}[\ensuremath{\mathbf{Proof}}]}
\newcommand{\bs}{\begin{proof}[\ensuremath{\mathbf{Solution}}]}
\newcommand{\ep}{\end{proof}}
\newcommand{\be}{\begin{equation}}
\newcommand{\ee}{\end{equation}}

\begin{document}

\def\Xint#1{\mathchoice
   {\XXint\displaystyle\textstyle{#1}}%
   {\XXint\textstyle\scriptstyle{#1}}%
   {\XXint\scriptstyle\scriptscriptstyle{#1}}%
   {\XXint\scriptscriptstyle\scriptscriptstyle{#1}}%
   \!\int}
\def\XXint#1#2#3{{\setbox0=\hbox{$#1{#2#3}{\int}$}
     \vcenter{\hbox{$#2#3$}}\kern-.5\wd0}}
\def\ddashint{\Xint=}
\def\dashint{\Xint-}

\title[Sticky particles with sticky boundary]{Asymptotics of the Sticky Particles evolution}


\author{Ryan Hynd}
\address{Department of Mathematics, University of Pennsylvania, Philadelphia, PA 19104}
\email{rhynd@math.upenn.edu}

\author{Adrian Tudorascu}
\address{Department of Mathematics, West Virginia University, Morgantown, WV  26506}
\email{adriant@math.wvu.edu}




\begin{abstract}
We study the long-time asymptotic behavior of the Sticky Particles dynamics on the real line. The time average of the Sticky Particles Lagrangian map has a limit which arises as a general property of projections onto closed convex cones in Hilbert spaces. More notably, we prove that the map itself has an asymptotic limit in the case where the Sticky Particles dynamics is confined to a compact set. \end{abstract}

\maketitle

\section{Introduction}
\subsection{Overview}
The general \emph{Pressureless Euler Equations} in one spatial dimension reads
\begin{equation}\label{PE}\tag{PE}
	\left\lbrace
	\begin{aligned}
		\partial_t \rho + \partial_x (\rho v) = 0 \\
		\partial_t (\rho v) + \partial_x (\rho v^2) = 0,\\
	\end{aligned}
	\right.
\end{equation}
where $\rho$ represents the mass density of the fluid at a particular location and time and $v$ represents the velocity of the fluid at a particular location and time.  We will restrict the space of allowable initial mass distributions to probability measures with finite second moment. This is a natural space, since in practical applications, the mass of the entire system is finite and so, without loss of generality, can be normalized to unity. The finite moment condition is also natural, as this only quantifies the fact that these mass distributions do not escape at infinity; moreover, prior works obtained existence of solutions by studying the system in the context of the Wasserstein space. 

Zeldovich \cite{Zeldovich1970} introduced the {\it Sticky Particles} model in order to provide a raw description of the formation of large scale structures in the universe.  The {\it Sticky Particles} model can be briefly described as follows.  If $m_i$, $i=1,...,n$ is a discrete system of masses initially located at $-\infty < x_1 < ... < x_n < +\infty$ and moving with initial velocities $v_i$, $i=1,...,n$, then one makes the assumption that the velocities remain constant while there is no collision.  At the collision of a group of particles, the particles stick together and the initial velocity of the newly formed particle is given by the conservation of momentum. It turns out the evolution of this system is described by (PE). In many of the works on existence of solutions \cite{Brenier1998}, \cite{weinan1996}, \cite{Tudorascu2008}, \cite{Tudorascu2015}, \cite{Hynd2019}, etc. the initial distribution is approximated by averages of Dirac masses and the ensuing Sticky Particles system is used to approximate solutions to (PE). 

The pressureless Euler system has been studied by different techniques in \cite{Brenier1998}, \cite{Bouchut1999}, \cite{Hynd2019}, \cite{Natile2009}, \cite{Tudorascu2008}, \cite{Tudorascu2015}, \cite{weinan1996}.  These techniques include the Sticky Particles model, description of the problem by an alternative scalar conservation law problem, and a semigroup approach. At the heart of our approach lies the identification (under appropriate initial conditions) of the Lagrangian solutions cf. Hynd \cite{Hynd2019} with the Scalar Conservation Laws solutions \cite{Brenier1998}, \cite{Tudorascu2008}, etc. We would also like to acknowledge the contribution in \cite{Natile2009}, where the solution is constructed by a Sticky Particles semigroup approach and a similar Lagrangian representation of the solution is obtained, albeit in a slightly weaker sense.

In a recent paper \cite{SuderTudor} it was shown that Lagrangian solutions as in \cite{Hynd2019} satisfy the strong initial continuity of the energy, which, together with the Oleinik condition, is a necessary ingredient for the uniqueness result. The latter was obtained by means of explicitly linking Hynd's Lagrangian solutions \cite{Hynd2019} and the Scalar Conservation Laws solutions (SCL solutions; see \cite{Brenier1998} and \cite{Tudorascu2008}) to the distributional solutions to a related problem whose well-posedness was shown in \cite{HuangWang}. More precisely, in \cite{SuderTudor} it is shown (that if $v_0$ is right-continuous and bounded on the real line, then any Lagrangian solution must coincide with the (unique) SCL solution. Besides the strong initial continuity of the energy condition (SICE), a stronger version of the Oleinik condition was necessary for this connection (with the solutions in \cite{HuangWang}) to work: the velocity $v$ admits a Borel representative $\hat v$ (i.e. $\hat v(t,\cdot)=v(t,\cdot)$, $\rho(t,\cdot)$-a.e. for all $t$) which satisfies a stronger version of the Oleinik condition, namely \eqref{eOleinik} ({\it everywhere} Oleinik), in the sense that the pertinent inequality is required to hold {\it everywhere} (and not just $\rho(t,\cdot)$-a.e.). This class of solutions is stable and contains the discrete Sticky Particles solutions \cite{SuderTudor}, so we call such solutions SPS (Sticky Particles Solutions).
 
In another recent paper \cite{TudorSPS}, the second author employed SPS and a reflection principle to solve the problem in the case where $\C$ is an arbitrary closed subset of the real line:  first, the case  $\C:=[0,\infty)$ (the case $(-\infty,0]$ follows by reflection) and that of a nondegenerate closed interval, say $\C:=[0,1]$, were analyzed. Then it is quite straightforward to glue the pieces together and produce a solution to the general problem. Of course, boundary conditions needed be imposed when $\partial\C$ is nonempty (that is, $\C\neq\Reals$) and in \cite{TudorSPS} it is successfully argued that among the two natural choices, i.e. reflective vs sticky boundary, only the latter leads to a well-posed problem. 

In the present paper we study the asymptotic behavior of SPS whose evolution is confined to a compact interval. It turns out that the solution $\rho$ converges to some limiting probability distribution $\rho_\infty$ as $t\rightarrow\infty$.

We conclude the introduction with some preliminary definitions and notation. In Section 2 we show that SPS solution on a bounded interval constructed in \cite{TudorSPS} can also be completely described in terms of the Lagrangian maps associated with the SPS for the Cauchy problem. In Section 3 we make the connection with the characterization of the Lagrangian map found in \cite{Natile2009}.  This allows for an obvious decay estimate for the time averaged Lagrangian map for the general Cauchy problem. In the case of a bounded interval, the Lagrangian map is uniformly (in space-time) bounded in $L^\infty$ so the said time average is trivially zero. The natural question then is if the map itself converges to a stationary map in the long-time asymptotic limit; we prove in Theorem \ref{main1} that it does. Section 4 deals with some open questions related to the rate of convergence to the equilibrium, while Section 5 presents an alternate proof for the existence of the asymptotic limit by employing the Lagrangian description from \cite{Hynd2019} with no recourse to the convex cone projection uncovered in \cite{Natile2009}.

\subsection{Preliminaries}
Several tools will be used in this paper, which are described below. Throughout this manuscript we denote by $t\in[0,\infty)$ the time variable and by $y\in\mathbb{R}$ the spatial variable; we reserve $x$ for the real numbers in the open interval $(0,1)$. This way, for any Borel probability measure $\rho$ on $\mathbb{R}$ we have the right-continuous optimal map $N=N(x)$ which pushes the Lebesgue measure restricted to $(0,1)$ forward to $\rho$, while its generalized inverse is the right-continuous cumulative distribution function of $\rho$, denoted by $M=M(y)$.  

\begin{notation}
The set of Borel probability measures on $\C$
will be denoted by $\mathcal{P}(\C)$. These objects can also be regarded as Borel probability measures on $\Reals$ whose support is a subset of $\C$.
\end{notation}

\begin{notation}
The set of Borel probability measures on $\C$ with finite second moment, i.e. those measures $\rho \in \mathcal{P}(\mathbb{R})$ such that 
\[
\mathrm{spt}(\rho)\subset\C\mbox{ and }\int_\mathbb{\mathbb{R}} y^2 \rho(dy) < \infty,
\]
will be denoted by $\mathcal{P}_2(\C)$.
\end{notation}

\begin{notation}
The set of bounded $\phi \in C^1 (\Reals)$ for which $\phi(y)=0$ for all $y\in\partial\C$ is denoted by $C^1_{0,b}(\C)$. If $\C=\Reals$, we identify $C^1_{0,b}(\C)\equiv C_b(\Reals)$. \\
The set of all bounded  $\phi \in C^1 ([0,\infty)\times\Reals)$ for which $\phi(t,\cdot)\in C^1_{0,b}(\C)$ for all $t\geq0$ is denoted by $C^1_{0,b}([0,\infty)\times\C)$.
\end{notation}

\begin{notation}
The set of bounded $\phi \in C^1 (\Reals)$ for which $\phi'(y)=0$ for all $y\in\partial\C$ is denoted by $C^1_{\nu,b}(\C)$. If $\C=\Reals$, we identify $C^1_{\nu,b}(\C)\equiv C_b(\Reals)$. \\
The set of all bounded  $\phi \in C^1 ([0,\infty)\times\Reals)$ for which $\phi(t,\cdot)\in C^1_{\nu,b}(\C)$ for all $t\geq0$ is denoted by $C^1_{\nu,b}([0,\infty)\times\C)$.
\end{notation}

\begin{definition}\label{bdry-solution}
A sequence $\{\rho_n\}_{n \in \mathbb{N}} \subset \mathcal{P}(\Reals)$ converges \emph{narrowly} to $\rho$ if
\[
\intR g d\rho_n \xrightarrow[n \rightarrow \infty]{} \intR g d\rho
\]
for each $g \in C_b(\Reals)$, where $C_b(\Reals)$ is the set of continuous, bounded functions on $\Reals$.
\end{definition}
It is easy to see that  $\{\rho_n\}_{n \in \mathbb{N}} \subset \mathcal{P}(\C)$ implies  $\rho \in\mathcal{P}(\C)$.
\begin{definition}\label{weak-sol-def}
Given $\rho_0 \in \mathcal{P}_2(\C)$, $v_0 \in L^2(\rho_0)$, a \emph{weak solution} to \eqref{PE} on $\C$ which satisfies the initial conditions $\rho(0,\cdot)=\rho_0$ and $v(0,\cdot)\rho(0,\cdot)=v_0\rho_0$ is a pair $(\rho,v)$ consisting of a narrowly continuous $\rho : [0, \infty) \rightarrow \mathcal{P}(\C)$ and a Borel map $v : [0,\infty)\times\C\rightarrow \mathbb{R}$ for which
\begin{enumerate}
 	\item For each $T > 0$,
	\[
	\int_0^T \intR v^2(t,y) \rho(t,dy) dt < \infty.
	\]
	\item For each $\phi \in C^1_{\nu,b}([0,\infty)\times\C)$,
	\[
	\int_0^\infty \intR (\partial_t \phi + v \partial_y \phi)(t,y) \rho(t,dy) dt + \intR \phi(0,y) \rho_0(dy) = 0.
	\]
	\item For each $\phi \in C^1_{0,b}([0,\infty)\times\C)$,
	\[
	\int_0^\infty \intR (v \partial_t \phi + v^2 \partial_y \phi) \rho(t,dy)dt + \intR \phi(0,y) v_0(y) \rho_0(dy) = 0.
	\]
\end{enumerate}
\end{definition}
\begin{definition}\label{OC}
Given a narrowly continuous $\rho : [0, \infty) \rightarrow \mathcal{P}(\C)$ and a Borel map $v : [0, \infty)\times\C\rightarrow\Reals$, we say $(\rho,v)$ satisfies the \emph{Oleinik condition} if for all $t>0$,
\begin{equation}\nonumber\label{eOleinik}\tag{{$e$Oleinik}}
\frac{v(t, y_2) - v(t, y_1)}{y_2 - y_1} \le \frac{1}{t}\mbox{ for all } y_1 < y_2\mbox{  in a connected component of }\C.
\end{equation}
\end{definition}
If we replace the domain $\C$ by $\Reals$, we get a condition on the whole $\Reals$. We note here that this is a stronger condition than what is usually called the Oleinik condition in the literature. Indeed, the standard formulation requires the pertinent inequality be satisfied $\rho(t,\cdot)$-a.e., instead of everywhere. 
 
\begin{definition}\label{SICE}
Let $\rho : [0, \infty) \rightarrow \mathcal{P}(\C)$ be narrowly continuous and let $v : [0, \infty)\times\C\rightarrow \Reals$ be a Borel map. We say $(\rho, v)$ satisfies the \emph{Strong Initial Continuity of Energy condition} (SICE) if $\forall \phi \in C_b(\mathbb{R})$:
\begin{equation}\nonumber\tag{{SICE}}\label{energy-cont}
\int_\mathbb{R} v^2(t,y) \phi(y) \rho(t,dy) \xrightarrow[t \rightarrow 0^+]{} \int_\mathbb{R} v_0^2(y) \phi (y)\rho_0(dy).
\end{equation}
\end{definition}
The main result of \cite{TudorSPS} is:
\begin{theorem}\label{super-main}
Let $\rho_0 \in \mathcal{P}_2(\C)$ and  $v_0\in C_0(\C)$  Then there exists a unique pair $(\rho,v)$ as in Definition \ref{weak-sol-def} which satisfies \eqref{eOleinik}, \eqref{energy-cont} and 
\begin{equation}\nonumber\tag{BC}\label{BCC}
v(t,y)=0\mbox{ for all }t\geq0\mbox{ and all }y\in\partial\C.
\end{equation}
\end{theorem}
A few remarks are in order: first, \eqref{BCC} ensures that ``nothing'' escapes the boundary once it reaches it. The condition seems stronger than necessary, as it may appear that it need not be imposed at every time $t\geq0$ and location $y\notin\mathrm{spt}(\rho_t)$. Secondly, note that uniqueness is to be understood in the following sense: if $(\rho_1,v_1)$ and $(\rho_2,v_2)$ are solutions as in Theorem \ref{super-main}, then $\rho_1(t,\cdot)=\rho_2(t,\cdot)=:\rho(t,\cdot)$ and $v_1(t,\cdot)=v_2(t,\cdot)=:\tilde v(t,\cdot),\ \rho(t,\cdot)$--a.e. for all $t\geq0$. Thus, the boundary condition \eqref{BCC} above only means $\tilde v$ has a Borel representative $v$ (i.e. $v(t,\cdot)=\tilde v(t,\cdot),\ \rho(t,\cdot)$--a.e. for all $t\geq0$) which satisfies \eqref{BCC}. 

\section{Lagrangian description of the SPS on an arbitrary closed set}
In this section we show that the SPS on $\C$ can be characterized fully in terms of the SPS of the Cauchy problem corresponding to the same initial data.

\begin{definition}
Given $\mu \in \mathcal{P}(\mathbb{R})$ and a Borel map $X : \mathbb{R} \rightarrow \mathbb{R}$, the \emph{push forward} $\nu$ of $\mu$ by $X$, denoted $\nu = X_{\#} \mu$ is a measure $\nu \in \mathcal{P}(\mathbb{R})$ defined as $\nu(B) = \mu(X^{-1}(B))$ for all Borel subsets $B$ of $\mathbb{R}$.
\end{definition}

A useful, fundamental result from measure theory, given here without proof, is the following:

\begin{prop}
Given $\mu, \nu \in \mathcal{P}(\mathbb{R})$ and a Borel map $X : \mathbb{R} \rightarrow \mathbb{R}$, 
\[
\nu = X_{\#} \mu \iff \int_\mathbb{R} \phi(X(y)) \mu(dy) = \int_\mathbb{R} \phi(y) \nu(dy), \forall \phi \in C_b(\mathbb{R}).
\]
\end{prop}



\begin{definition}\label{cond-expect}
Given $\rho_0 \in \mathcal{P}_2(\mathbb{R})$, $v_0 \in L^2(\rho_0)$, and a Borel map $X : \mathbb{R} \rightarrow \mathbb{R}$, the \emph{conditional expectation of $v_0$ with respect to $\rho_0$ given $X$} is $f\circ X$, where $f \in L^2(X_{\#}\rho_0)$ is unique (guaranteed by the Riesz-Fr\'{e}chet Representation Theorem on the Hilbert space $L^2(X_{\#} \rho_0)$) such that:
\[
\int_\mathbb{R} f(X(y)) \zeta(X(y)) \rho_0(dy) = \int_\mathbb{R} v_0(y) \zeta(X(y)) \rho_0(dy)
\]
for all $\zeta \in L^2(X_{\#} \rho_0)$.
\end{definition}

The conditional expectation of $v_0$ with respect to $\rho_0$ given $X$ is denoted by $\mathbb{E}_{\rho_0} [v_0 | X]$.

In this paper, we make use of the \emph{Sticky Particles Flow Equation with Initial Condition} formulation as in \cite{Hynd2019}:

\begin{equation}\label{SPF-with-IC}\tag{SPF-IC}
	\left\lbrace
	\begin{aligned}
	   \dot{X}(t,\cdot) = \mathbb{E}_{\rho_0} [v_0 | X(t,\cdot)], \text{ a.e. } t \ge 0 \\
	   X(0,\cdot) = \mathrm{id}_\mathbb{R}.\\
	\end{aligned}
	\right.
\end{equation}
If $v_0$ is right-continuous and bounded on $\Reals$, it is proved in \cite{SuderTudor} that \eqref{SPF-with-IC} admits a solution $X$ which is jointly Borel, absolutely continuous as a map $X:[0,\infty)\rightarrow L^2(\rho_0)$, and that there exists a Borel map $v:[0,\infty)\times \mathbb{R}\rightarrow \mathbb{R}$ such that
\begin{equation}\label{cond-expect-form}
    \mathbb{E}_{\rho_0} [v_0 | X(t,\cdot)]=v(t,X(t,\cdot)).
\end{equation} 
Furthermore, if $\rho(t,\cdot):=X(t,\cdot)_\#\rho_0$,
then $(\rho, v)$ is a distributional solution to (PE-IC). It is proved in \cite{SuderTudor} that $v$ admits a Borel representative, still denoted by $v$, which satisfies \eqref{eOleinik} and \eqref{energy-cont}, which yields uniqueness by a remarkable result in \cite{HuangWang}. This solution is termed the {\it sticky particles solution} (or SPS) because it coincides with the Scalar Conservation Law solution, which is obtained from the Sticky Particles paradigm \cite{Brenier1998}. 

\noindent {\bf A modified flow-map and velocity.}
We next show that the SPS on $\C$ given by Theorem \ref{super-main} admits an explicit Lagrangian representation in terms of $X$. More precisely, if $(-\infty,a]$ is a connected component of $\C$, we consider $X$ the Lagrangian map of the SPS to the Cauchy problem corresponding to the initial data $(\rho_0|_{(-\infty,a]},v_0|_{(-\infty,a]})$ and define 
\begin{equation}\label{X-ray-a}
Y(t,y):=X(t,y)\mbox{ if }X(t,y)<a\mbox{ and }Y(s,y)=a\mbox{ for all }s\geq t\mbox{ if }X(t,y)=a.
\end{equation}
Moreover, consider the Borel velocity $v$ constructed in \cite{SuderTudor} (which satisfies the {\it everywhere} Oleinik condition) for this initial data and set  
\begin{equation}\label{v1}
\tilde v(t,y):=v(t,y)\mbox{ if }y<a\mbox{ and }\tilde v(t,a)=0.
\end{equation}
 If, on the other hand, $[b,\infty)$ is a connected component of $\C$, we consider $X$ the Lagrangian map of the SPS to the Cauchy problem corresponding to the initial data $(\rho_0|_{[b,\infty)},v_0|_{[b,\infty)})$ and define 
 \begin{equation}\label{X-ray-b}
Y(t,y):=X(t,y)\mbox{ if }X(t,y)>b\mbox{ and }Y(s,y)=b\mbox{ for all }s\geq t\mbox{ if }X(t,y)=b.
\end{equation}
 \begin{figure}[h]
\scalebox{0.95}
 {\includegraphics[width=\textwidth]{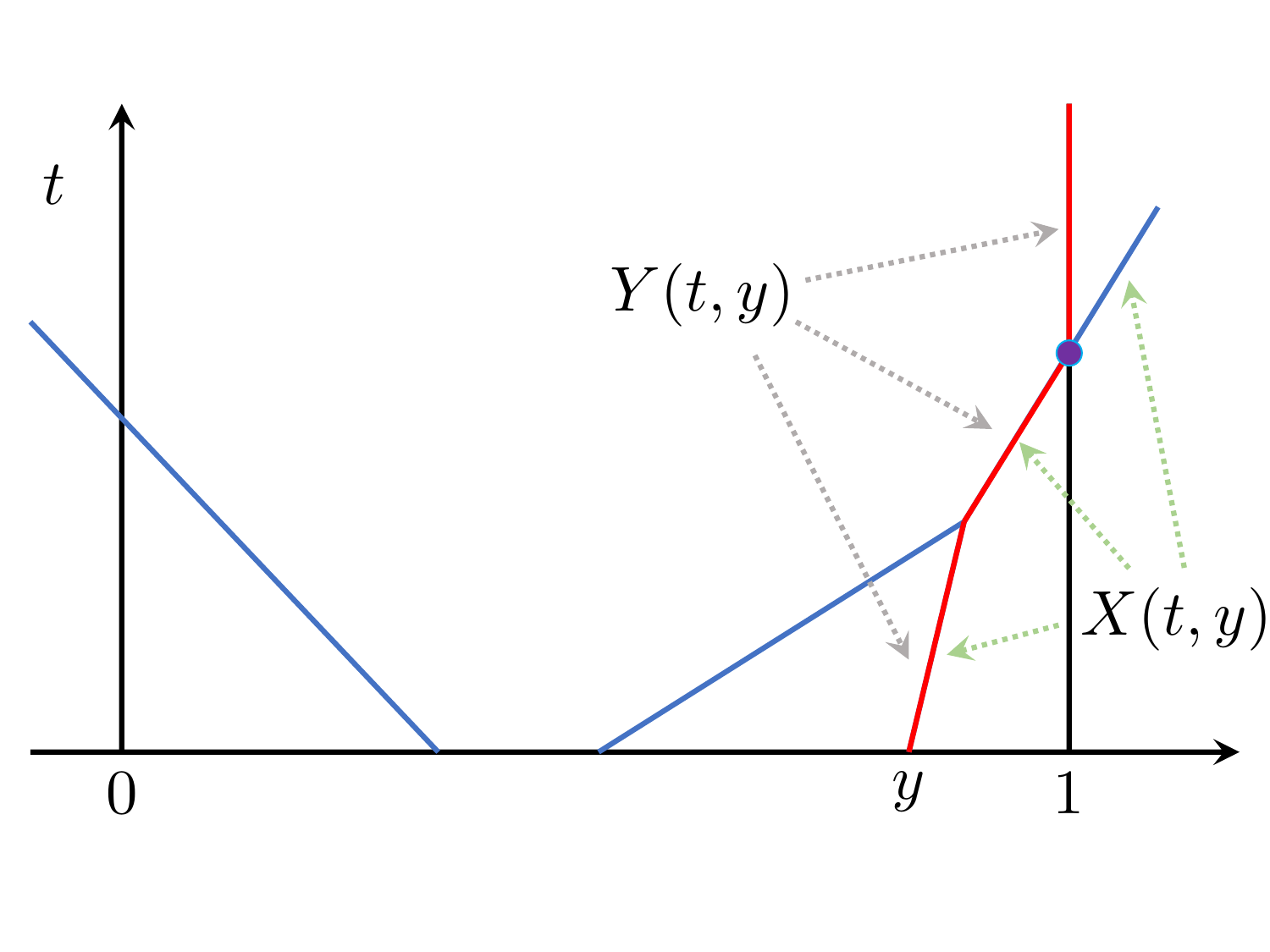}}
  \caption{{Example of trajectories $X$ and $Y$ originating at $y$, which is in the support of a discrete probability measure supported at three points within $(0,1)$. The node on the vertical line $y = 1$ is the time when $X(\cdot·,y)$ first hits this line.}}
  \label{states1}
\end{figure}
Now consider the Borel velocity $v$ constructed in \cite{SuderTudor} for this new initial data and set  
\begin{equation}\label{v2}
\tilde v(t,y):=v(t,y)\mbox{ if }y>b\mbox{ and }\tilde v(t,b)=0.
\end{equation}
Finally, if $[a,b]$ is a nondegenerate connected component of $\C$, we take $X$ the Lagrangian map of the SPS to the Cauchy problem corresponding to the initial data $(\rho_0|_{[a,b]},v_0|_{[a,b]})$ and define 
\begin{eqnarray}
Y(t,y):=X(t,y)\mbox{ if }a<X(t,y)<b,\nonumber\\
 \label{X-ray-ab}Y(s,y)=a\mbox{ for all }s\geq t\mbox{ if }X(t,y)=a,\\
 Y(s,y)=b\mbox{ for all }s\geq t\mbox{ if }X(t,y)=b.\nonumber
\end{eqnarray}
Obviously, if $v$ is the one from \cite{SuderTudor} for the latter initial data, we shall define 
\begin{equation}\label{v3}
\tilde v(t,y):=v(t,y)\mbox{ if }a<y<b\mbox{ and }\tilde v(t,a)=0=\tilde v(t,b).
\end{equation}
Then we have:
\begin{theorem}\label{tildeX-X}
Let $\rho_0 \in \mathcal{P}_2(\C)$ and  $v_0\in C_0(\C)$  and let $Y$ be defined as in \eqref{X-ray-a}, \eqref{X-ray-b}, \eqref{X-ray-ab} and $\tilde v$ as in \eqref{v1}, \eqref{v2}, \eqref{v3}. Set $\tilde\rho(t,\cdot):=Y(t,\cdot)_\#\rho_0$ for all $t\geq0$. Then $(\tilde\rho,\tilde v)$ is the SPS on $\C$ (i.e., as in Theorem \ref{super-main}).
\end{theorem}
\begin{proof} It suffices to prove this theorem in the case where $\C$ is connected. The most delicate case is $\C=[a,b]$ for some real numbers $a<b$ so we shall focus on this one; it will be clear to the reader what modifications are required in order to handle the case of a closed ray. 

Without loss of generality, let us take $a=0,\ b=1$. If $\rho_0(\{0\})>0$, then $Y(t,0)=0$ for all $t\geq0$, whereas $Y(t,1)=1$ for all $t\geq0$ if $\rho_0(\{1\})>0$. So let us assume $\rho_0$ gives mass to neither 0 nor 1. Then there is a Borel set $B\subset(0,1)$ such that $(0,1)\setminus B$ is $\rho_0$--negligible and $X(\cdot,y)\in W_{loc}^{1,\infty}([0,\infty))$ for all $y\in B$; moreover, the weak time derivative of $X$ is $v(\cdot,X)$, i.e.
$$X(t,y)-X(s,y)=\int_s^tv(\tau, X(\tau,y))d\tau\mbox{ for all }y\in B\mbox{ and all }0\leq s<t.$$
If $0<X(t,y)<1$, then $0<X(\tau,y)<1$ for all $0\leq\tau\leq t$ and so $Y(\tau,y)= X(\tau,y)$ and $\tilde v(\tau,Y(\tau,y))=v(\tau,X(\tau,y))$ for all $0\leq\tau\leq t$. Thus,
\begin{equation}\label{X-Xdot}
Y(t,y)-Y(s,y)=\int_s^t\tilde v(\tau,Y(\tau,y))d\tau\mbox{ for all }y\in B\mbox{ and all }0\leq s<t.
\end{equation}
If $X(t,y)=0$, then there exists $0<t_0\leq t$ such that $X(\tau,y)>0$ for all $0<\tau<t_0$ and $X(t_0,y)=0$. This means $Y(\tau,y)=0$ for all $t_0\leq\tau\leq t$. If $s<t_0$, we have $1>Y(s,y)=X(s,y)>0$, so
$$Y(t,y)-Y(s,y)=X(t_0,y)-X(s,y)=\int_s^{t_0}v(\tau,X(\tau,y))d\tau=\int_s^{t}\tilde v(\tau,Y(\tau,y))d\tau,$$
since $\tilde v(\tau,Y(\tau,y))=0$ for $t_0\leq\tau\leq t$. A similar argument applies if $X(t,y)=1$, so \eqref{X-Xdot} holds for all $y\in B$  and all $0\leq s<t$.
This is equivalent to $(\tilde\rho,\tilde v)$ solving the continuity equation on $\C$.  
Next, take $\varphi\in C(\Reals)$ supported in $[0,1]$ and note that
\begin{eqnarray*}
\int_{\Reals}\varphi(Y(t,y))\tilde v(t,Y(t,y))\rho_0(dy)&=&\int_{\Reals}\varphi(X(t,y))v(t,X(t,y))\rho_0(dy)\\
&=&\int_{\Reals}v_0(y)\varphi(X(t,y))\rho_0(dy)\\
&=&\int_{\Reals}v_0(y)\varphi(Y(t,y))\rho_0(dy).
\end{eqnarray*}
This implies 
$$\int_{\Reals}\varphi(y)\tilde v(t,y)\tilde\rho(t,dy)=\int_{\Reals}v_0(y)\varphi(Y(t,y))\rho_0(dy),$$
so the momentum equation holds (as in Definition \ref{weak-sol-def}). 

Likewise, if $\varphi\in C(\Reals)$ is supported in $[0,1]$, we have 
\begin{eqnarray*}
\int_{\Reals}\varphi(Y(t,y))\tilde v^2(t,Y(t,y))\rho_0(dy)&=&\int_{\Reals}\varphi(X(t,y))v^2(t, X(t,y))\rho_0(dy)\\
&=&\int_{\Reals}v_0(y)\varphi( X(t,y))v(t,X(t,y))\rho_0(dy)\\
&=&\int_{\Reals}v_0(y)\varphi(Y(t,y))\tilde v(t,Y(t,y))\rho_0(dy).
\end{eqnarray*}
It is not hard to show that the latter converges to $\int_{\Reals}v_0^2(y)\varphi(y)\rho_0(dy)$, so \eqref{energy-cont} is satisfied.

Since $\rho_0$ and $v_0$ are supported in $[0,1]$, if we approximate $\rho_0$ by averages of Dirac delta's $\rho_{0,n}$ supported in $[0,1]$, it is obvious that $v_n\leq0$ on $[0,\infty)\times(-\infty,0]$ while $v_n\geq0$ on $[0,\infty)\times[1,\infty)$. Due to the stability of SPS (i.e. the narrow convergence of $v_n\rho_n$ to $v\rho$); see, e.g., \cite{SuderTudor}, we conclude $v\leq0$ on $[0,\infty)\times(-\infty,0]$ while $v\geq0$ on $[0,\infty)\times[1,\infty)$. Since $v$ satisfies the everywhere Oleinik condition, we get, in particular  
$$\frac{v(t,y_2)-v(t,y_1)}{y_2-y_1}\leq\frac{1}{t}\mbox{ for all }0\leq y_1<y_2\leq 1\mbox{ and all }t>0.$$
But $v(t,0)\leq 0$ while $v(t,1)\geq 0$, so the above inequality further implies 
\begin{equation}\label{veloc-ineq}
\frac{-v(t,y_1)}{1-y_1}\leq\frac{1}{t}\mbox{ and }\frac{v(t,y_2)}{y_2}\leq\frac{1}{t}\mbox{ for all }0\leq y_1<y_2\leq 1\mbox{ and all }t>0.
\end{equation}
Thus, $\tilde v$ defined in \eqref{v1}, \eqref{v2}, \eqref{v3} satisfies \eqref{eOleinik}. 
\end{proof}
Next let 
$$\mathcal{I}:=\{y\in B\,:\, 0<Y(t,y)\mbox{ for all }t\geq0\}$$
and
$$\mathcal{S}:=\{y\in B\,:\, Y(t,y)<1\mbox{ for all }t\geq0\}.$$
Note that $\mathcal{I}=\emptyset$ means that all mass concentrates at $0$ as $t\rightarrow\infty$, whereas $\mathcal{S}=\emptyset$ means all mass concentrates at $1$ as $t\rightarrow\infty$. Therefore, the interesting case is $\mathcal{I}\neq\emptyset\neq\mathcal{S}$. If $\alpha:=\inf\mathcal{I}$ and $\beta:=\sup\mathcal{S}$, in the latter case we must have $0\leq\alpha\leq\beta\leq1$. If $0=\alpha=\beta$, we infer $\mathcal{I}=\emptyset$, a contradiction; if $\alpha=\beta=1$, we infer $\mathcal{S}=\emptyset$, yet another contradiction. This means that if $\alpha=\beta$, then $0<\alpha<1$ and the trajectories starting in $(0,\alpha)$ concentrate asymptotically at 0, while the ones starting in $(\beta,1)$ concentrate asymptotically at 1. If $\alpha\in B$, then either $Y(t_0,\alpha)=0$ or $Y(t_0,\alpha)=1$ for some $t_0>0$, or $Y(t,\alpha)=\alpha$ for all $t\geq0$. Finally, if $\alpha<\beta$, we have that all mass originating in $(0,\alpha)$ concentrates asymptotically at 0, while the mass originating in $(\beta,1)$ does so at 1. 

If $\alpha\in B$, then either $\alpha\in\mathcal{I}$ or $\alpha\in B\setminus\mathcal{I}$. If the former, then $Y(t,\alpha)\geq\alpha$ for all $t\geq0$; if the latter, then there exists a minimal $t_0>0$ such that $Y(t,\alpha)=0$ for all $t\geq t_0$.  We reach similar conclusions if $\beta\in B$. Finally, if $y\in B\cap(\alpha,\beta)$, then $\alpha<Y(t,y)<\beta$ for all $t\geq0$. In other words, $Y(t,(\alpha,\beta))\subset(\alpha,\beta)$ for all $t\geq0$. From the definition of $Y$, we also deduce $Y=X$ on $[0,\infty)\times(\alpha,\beta)$. 
This leads to the realization that the study of the asymptotic behavior of the SPS on a nondegenerate compact interval $[a,b]$ (which includes the support of $\rho_0$) can be reduced to the study of the SPS on $\Reals$ for the same $\rho_0$ and initial velocity extended by zero outside $[a,b]$, {\it under the assumption} that $X(t,[a,b])\subset[a,b]$ for all $t\geq0$.  
 \begin{figure}[h]
\scalebox{0.90}
 {\includegraphics[width=\textwidth]{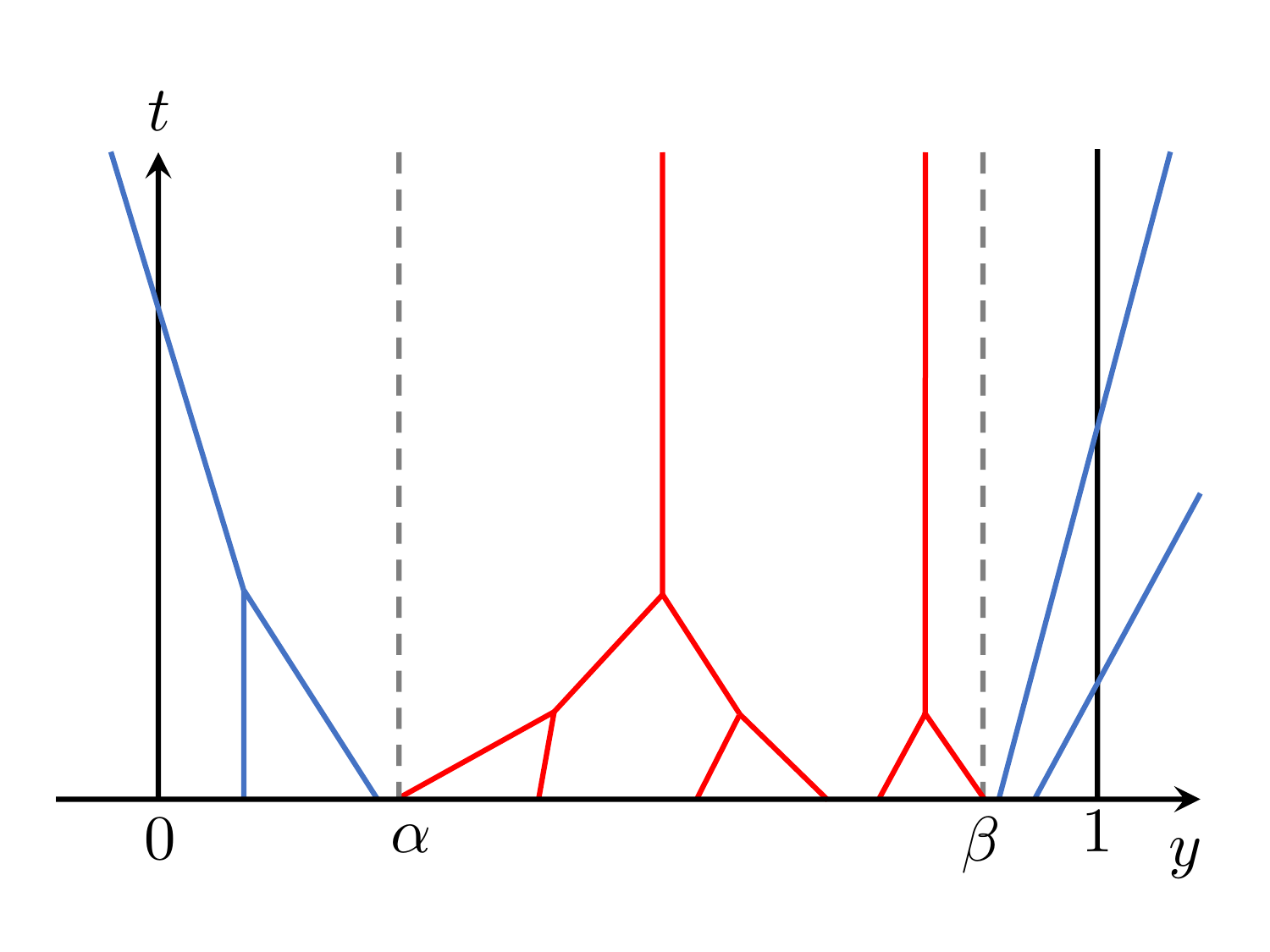}}
  \caption{{This is an example of the sticky particles evolution in which particles that start out in the interval $[\alpha,\beta]$ remain confined to this interval for all later times.}}
  \label{states3}
\end{figure}

To further reformulate and to preface the next section, the question now is: if the sticky particles flow originating from a (probability) distribution supported in a compact interval is non-expanding (i.e. $X(t,[a,b])\subset[a,b]$ for all $t\geq0$), does the flow map $t\rightarrow X(t,\cdot)$ have a limit as $t\rightarrow\infty$?   

\section{Asymptotic behavior of SPS confined to compact intervals}
Here we start from the obvious connection between the SPS for the Cauchy problem and the semigoup solutions constructed by Natile and Savare \cite{Natile2009}. First, we remind the reader of a result from \cite{SuderTudor} which shows that $X$ can be chosen such that $N(t,\cdot):=X(t,\cdot)\circ N_0$, where $N(t,\cdot)$ is the right continuous optimal map pushing the Lebesgue measure restricted to $[0,1]$ forward to $\rho(t,\cdot)$ and $N_0:=N(0,\cdot)$. Let $\mathcal{K}$ be the set of nondecreasing functions in  $L^2(0,1)$. Then $\mathcal{K}$ is a closed, convex cone. Without loss of generality, we may assume the elements of $\mathcal{K}$ to be right continuous (we can single out a unique right-continuous representative since these are all monotone functions). In \cite{Natile2009} it was established that 
$$N(t,\cdot)=\mathrm{proj}_{\mathcal{K}}(N_0+tV_0),\mbox{ where }V_0:=v_0\circ N_0$$ for all solutions arising as limits of discrete SPS, so it holds for the SPS when $v_0$ is continuous and bounded (see \cite{SuderTudor}). The contraction and positive homogeneity properties of the projection onto a closed convex cone $\mathcal{K}$ in any Hilbert space $H$ reveal 
$$\|\mathrm{proj}_{\mathcal{K}}(u_0+tw_0)-t\mathrm{proj}_{\mathcal{K}}w_0\|\leq \|u_0\|\mbox{ for all }u_0,\ w_0\in H,\ t\geq0,$$
which implies 
\begin{equation}\label{generic-decay}
\Big\|\frac{N(t,\cdot)}{t}-\mathrm{proj}_{\mathcal{K}}V_0\Big\|_2\leq\frac{\|N_0\|_2}{t}\mbox{ for all }t>0,
\end{equation}
where $\|\cdot\|_p$ denotes the $L^p$-norm on $[0,1]$.
If $\rho_0=\delta_{1/2}$ and $v_0(1/2)=1$ we have $N_0\equiv1/2$, $V_0\equiv1\equiv\mathrm{proj}_{\mathcal{K}}V_0$, $N(t,\cdot)=1/2+t$ and we get equality in \eqref{generic-decay} for all $t>0$. Thus, \eqref{generic-decay} is optimal. Ultimately, the same example also shows that we cannot expect a better decay estimate which will guarantee that $\|N(t,\cdot)-t\mathrm{proj}_{\mathcal{K}}V_0\|_2$ converges to zero.

And yet, while  \eqref{generic-decay} spells out the optimal decay for generic SP solutions, it does not offer any insight into what happens if the SP evolution is confined to a compact interval. Whereas $\mathrm{spt}(\rho(t,\cdot))\subset[0,1]$  is equivalent to $N(t,[0,1])\subset[0,1]$ for all $t\geq0$, we can only infer from \eqref{generic-decay} that in this case $N(t,\cdot)/t$ decays uniformly to zero and 
\begin{equation}\label{null-proj}
\mathrm{proj}_{\mathcal{K}}V_0\equiv0.
\end{equation}
 The latter can be regarded as a necessary condition for the SP evolution to remain confined to $[0,1]$. In fact, 
\begin{equation}\label{unif-L1}
\sup_{t\geq0}\int_\Reals|y|\rho(t,dy)<\infty\ (\mbox{i.e. }\{N(t,\cdot)\}_{t\geq0}\mbox{ is bounded in }L^1(0,1))
\end{equation}
 is sufficient to deduce \eqref{null-proj}. 
 \begin{lemma}\label{V0F0}
Let $\rho_0\in\mathcal{P}_2(\Reals)$, $v_0\in C_b(\Reals)$ and define
$$F_0(x):=\int_0^xV_0(z)dz\mbox{ for all }0\leq x\leq 1.$$
Then \eqref{null-proj} is equivalent to 
$$F_0(x)\geq F_0(1)=0\mbox{ for all }0\leq x\leq 1.$$
\end{lemma}
\begin{proof}
We know (see, e.g., \cite{Natile2009}) that $\mathrm{proj}_{\mathcal{K}}V_0=\hat F_0'$, where $\hat F_0$ is the convex envelope (Lipschitz, because $V_0\in L^\infty(0,1)$; again, see \cite{Natile2009}) of $F_0$ over $[0,1]$. In other words, $\hat F_0$ the largest convex function pointwise smaller than $F_0$ over $[0,1]$ and with $\hat F_0(0)=F_0(0)$ and $F_0(1)=\hat F_0(1)$. Since $F_0(0)=0$, \eqref{null-proj} is equivalent to $\hat F_0\equiv0$ on $[0,1]$, which is the same as $F_0(x)\geq F_0(1)=0\mbox{ for all }0\leq x\leq 1.$
\end{proof}
We are now ready to prove:
\begin{prop}\label{prop-comp-conf}
Let $\rho_0\in\mathcal{P}_2(\Reals)$  and let $v_0\in C_b(\Reals)$ be such that the SPS $(\rho,v)$ corresponding to the initial data $(\rho_0,v_0)$ satisfies
$$\mbox{ for each }x\in[0,1],\ \inf_{t\geq0}\int_0^xN(t,z)dz>-\infty$$
and
$$\sup_{t\geq0}\int_0^1N(t,x)dx<\infty,\mbox{ i.e. }\sup_{t\geq0}\int_\Reals y\rho(t,dy)<\infty.$$
Then \eqref{null-proj} holds.
\end{prop}
\begin{proof}
Another generic property of the projection onto closed convex cones is 
\begin{equation}\label{gener-ineq}
\langle u_0+tw_0-\mathrm{proj}_{\mathcal{K}}(u_0+tw_0),\kappa\rangle\leq0\mbox{ for all }t\geq0\mbox{ and }u_0,\ w_0\in H,\ \kappa\in \mathcal{K}.
\end{equation}
Since all constant functions belong to our particular $\mathcal{K}$ (thus, $f\equiv1$ and $g\equiv-1$, in particular), we infer 
$$\int_0^1N_0(x)dx+t\int_0^1V_0(x)dx=\int_0^1N(t,x)dx\mbox{ for all }t\geq0.$$
The hypotheses imply
\begin{equation}\label{generic-conserve1}
\int_0^1V_0(x)dx=0,\mbox { i.e. }\int_\Reals v_0(y)\rho_0(dy)=0 
\end{equation}
or, equivalently,
\begin{equation}\label{generic-conserve2}
\int_0^1N(t,x)dx=\int_0^1N_0(x)dx,\mbox{ i.e. }\int_\Reals y\rho_0(dy)=\int_\Reals y\rho(t,dy)\mbox{ for all }t>0.
\end{equation}
Now take $\kappa:=-1_{[0,x]}$ for some arbitrary $0<x\leq1$ in \eqref{gener-ineq} to deduce 
$$\int_0^x N_0(z)dz+t\int_0^xV_0(z)dz\geq\int_0^xN(t,z)dz\mbox{ for all }t\geq0$$ 
so the first hypothesis implies 
$F_0(x)\geq0\mbox{ for all }x\in[0,1].$
The conclusion now follows from \eqref{generic-conserve1} and Lemma \ref{V0F0}.
\end{proof}
\begin{remark}\label{rmk-proj0}
It is not difficult to see that Sticky Particles evolutions confined to $[0,1]$, i.e. $\mathrm{spt}(\rho(t,\cdot))\subset[0,1]$ for all $t\geq0$, satisfy the hypotheses of Proposition \ref{prop-comp-conf}.
\end{remark}
Let us now consider $\rho_0\in\mathcal{P}(\Reals)$ with $\mathrm{spt}(\rho_0)\subset[0,1]$, $v_0\in C(\Reals)$ with $v_0(0)=0=v_0(1)$.
\begin{lemma}\label{V0Nt}
The map $\theta:[0,\infty)\rightarrow\Reals$ given by 
$$\theta(t):=\int_\Reals  V_0(x)N(t,x)dx\mbox{ is nondecreasing.}$$ If the SP evolution is confined to a compact interval, then 
$$\lim_{t\rightarrow\infty}\theta(t)=0.$$
\end{lemma}
\begin{proof} From \cite{SuderTudor} we have $|v(t,y)|\leq\sup|v_0|$ for all $(t,y)\in[0,\infty)\times\Reals$ and $\dot N(t,x)=v(t,N(t,x))$ for a.e. $(t,x)\in[0,\infty)\times[0,1]$.
The map from the statement is locally Lipschitz and its derivative is, according to \eqref{cond-expect-form},
$$\int_0^1 V_0(x)v(t,N(t,x))dx=\int_0^1v^2(t,N(t,x))dx=\int_\Reals v^2(t,y)\rho(t,dy),$$
so that proves its monotonicity.

Yet another property of the projection onto a convex cone is 
\begin{equation}\label{gener-eq}
\langle u_0+tw_0-\mathrm{proj}_{\mathcal{K}}(u_0+tw_0),\mathrm{proj}_{\mathcal{K}}(u_0+tw_0)\rangle=0\mbox{ for all }t\geq0\mbox{ and }u_0,\ w_0\in H.
\end{equation}
In our case, this translates to
\begin{equation}\label{gener-eq-Nt}
\int_0^1N_0(x)N(t,x)dx+t\theta(t)=\int_0^1N^2(t,x)dx.
\end{equation}
In the confined case, $N$ is a bounded map (in time and space), so \eqref{gener-eq-Nt} implies 
\begin{equation}\label{0-0}
\theta(t)\leq0=\lim_{t\rightarrow\infty}\theta(t)\mbox{ for all }t\geq0.
\end{equation}
The inequality is due to the monotonicity of $\theta$.
\end{proof}
We shall prove the following result:
\begin{theorem}\label{main1}
Let $\rho_0 \in \mathcal{P}(\Reals)$ and  $v_0\in C(\Reals)$, both supported in $[0,1]$.  If the SP evolution is confined to $[0,1]$, then there exists a right-continuous, nondecreasing function $N_\infty:[0,1]\rightarrow[0,1]$ such that 
\begin{equation}\label{Nt-cv}
\lim_{t\rightarrow\infty}N(t,x)=N_\infty(x)\mbox{ for a.e. }x\in[0,1].
\end{equation}
\end{theorem}
\begin{proof}
Let $\{t^{(1)}_n\}_n$ be an arbitrary sequence of times going to infinity. By Helly's Selection Theorem, there exists a subsequence (not relabelled) and a limiting function $f$, right continuous and nondecreasing such that $N(t^{(1)}_n,\cdot)$ converges to $f$ pointwise Lebesgue a.e. in $[0,1]$. Since all functions involved take values in $[0,1]$, we infer that, in particular, $N(t^{(1)}_n,\cdot)$ converges to $f$ in $L^2(0,1)$. Likewise, we get a limit $g$ for another arbitrary sequence $\{t^{(2)}_n\}_n$ which diverges to infinity.
From \eqref{gener-ineq} we deduce
$$\int_0^1N_0(x)f(x)dx+t_n^{(2)}\int_0^1V_0(x)f(x)dx\leq\int_0^1N(t^{(2)}_n,x)f(x)dx$$
and
$$\int_0^1N_0(x)g(x)dx+t_n^{(1)}\int_0^1V_0(x)g(x)dx\leq\int_0^1N(t^{(1)}_n,x)g(x)dx.$$
From \eqref{0-0} we get
$$\int_0^1V_0(x)f(x)dx=\int_0^1V_0(x)g(x)dx=0,$$
so the previous displayed inequalities imply (in the limit)
$$\int_0^1N_0(x)f(x)dx\leq\int_0^1f(x)g(x)dx\mbox{ and }\int_0^1N_0(x)g(x)dx\leq\int_0^1f(x)g(x)dx.$$
Finally, \eqref{gener-eq-Nt} and \eqref{0-0} along $t=t_n^{(1)}$ and $t=t_n^{(2)}$, respectively, imply (in the limit as $n\rightarrow\infty$) 
$$\int_0^1N_0(x)f(x)dx\geq\int_0^1f^2(x)dx\mbox{ and }\int_0^1N_0(x)g(x)dx\geq\int_0^1g^2(x)dx.$$
The last four displayed inequalities imply $\|f-g\|_2\leq0$. So, all the sub-sequential asymptotic  limits coincide, which finishes the proof.
\end{proof}
\begin{remark}
To summarize, we have proved that SP solutions to the Cauchy problem have an asymptotic limit if they are confined to a compact interval. Consequently, if $\C$ is a compact subset of $\Reals$, then the SP solution on $\C$ has an asymptotic limit, provided that $v_0\in C_0(\C)$.
\end{remark}
The proof of Theorem \ref{main1} also reveals
$$\int_0^1N_0(x)N_\infty(x)dx=\int_0^1N^2_\infty(x)dx.$$

Since for all $0\leq s\leq t$ we have \cite{Natile2009} $N(t,\cdot)=\xi_{s,t}\circ N(s,\cdot)$ for some Borel function $\xi_{s,t}$, we deduce $N_\infty=\xi_t\circ N(t,\cdot)$ for some Borel function $\xi_t$. Consequently, 
$$\frac{d}{dt}\int_0^1N(t,x)N_\infty(x)dx=\int_0^1v(t,N(t,x))N_\infty(x)dx=\int_0^1v(t,N(t,x))\xi_t(N(t,x))dx$$
$$=\int_0^1V_0(x)\xi_t(N(t,x))dx=\int_0^1V_0(x)N_\infty(x)dx=0.$$
We infer that 
$$\int_0^1N(t,x)N_\infty(x)dx=\int_0^1N^2_\infty(x)dx\mbox{ for all }t\geq0,$$
so that 
$$e(t):=W_2^2(\rho(t,\cdot),\rho_\infty)=\|N(t,\cdot)-N_\infty\|^2_2=\|N(t,\cdot)\|_2^2-\|N_\infty\|_2^2.$$
where $W_2$ is the Wasserstein distance with quadratic cost \cite{AGS}.
We then have 
\begin{equation}\label{derivative-1}
\dot e(t)=2\int_0^1N(t,x)v(t,N(t,x))dx=2\int_0^1V_0(x)N(t,x)dx,
\end{equation}
which, by Lemma \ref{V0Nt}, means $e$ is nonincreasing with derivative converging asymptotically to zero. Furthermore, from the proof of Lemma \ref{V0Nt}, we deduce 
\begin{equation}
\label{derivative-2}
\ddot e(t)=2\int_\Reals v^2(t,y)\rho(t,dy)=2|\rho'|^2(t).
\end{equation}
Here $|\rho'|$ denotes the {\it metric derivative} \cite{AGS} of the absolutely continuous curve $t\mapsto\rho(t,\cdot)$, defined for a.e. $t>0$ as
$$|\rho'|(t):=\lim_{s\rightarrow t}\frac{W_2(\rho(s,\cdot),\rho(t,\cdot))}{|s-t|}.$$
It is the smallest function $\beta\in L_{loc}^2(0,\infty)$ such that
$$W_2(\rho(s,\cdot),\rho(t,\cdot))\leq\int_s^t\beta(\tau)d\tau\mbox{ for all }0\leq s\leq t.$$
Thus, $e$ is convex on $[0,\infty)$ and asymptotically decreasing to zero. From \eqref{derivative-1} and \eqref{derivative-2} we first infer
\begin{equation}\label{asymp1}
|\rho'|\in L^2(0,\infty)\mbox{ and }2\int_0^\infty|\rho'|^2(t)dt=-\langle V_0,N_0\rangle 
\end{equation}
and 
\begin{equation}\label{asymp2}
\mathrm{id}\cdot|\rho'|^2\in L^1(0,\infty)\mbox{ and }2\int_0^\infty t|\rho'|^2(t)dt=W_2^2(\rho_0,\rho_\infty). 
\end{equation}
In fact, we get something more general than \eqref{asymp2}, namely
\begin{equation}\label{asymp3}
e(t)=2\int_t^\infty (s-t)|\rho'|^2(s)ds=2\int_0^\infty s|\rho'|^2(s+t)ds\mbox{ for all }t\geq0. 
\end{equation}
Since the energy is nonincreasing along SPS, we can actually improve on the obvious inequality 
$$|\rho'|(t)=\|v(t,\cdot)\|_{L^2(\rho(t,\cdot))}\leq \frac{1}{t},$$
which in the confined case arrives from the \eqref{eOleinik} condition. Indeed, by \eqref{veloc-ineq}, we have
$$\max\bigg\{\frac{v(t,y)}{y},\frac{-v(t,y)}{1-y}\bigg\}\leq\frac{1}{t},$$
which is equivalent to 
\begin{equation}\label{veloc-ineq1}
\frac{y-1}{t}\leq v(t,y)\leq\frac{y}{t}\mbox{ for all }y\in[0,1],\ t>0.
\end{equation}
In particular, this means
\begin{equation}\label{unif-velo-ineq}
|v(t,y)|\leq\frac{1}{t}\mbox{ for all }y\in[0,1],\ t>0
\end{equation}
\begin{prop}\label{metric-deriv-decay}
Let $\rho_0 \in \mathcal{P}(\Reals)$ and  $v_0\in C(\Reals)$, both supported in $[0,1]$.  If the SP evolution is confined to $[0,1]$, then
\begin{equation}\label{better-velo-decay}
\lim_{t\rightarrow\infty}t|\rho'|(t)=0.
\end{equation}
\end{prop}
\begin{proof}
From \eqref{asymp2} we get
$$\lim_{t\rightarrow\infty}\int_t^{2t}s|\rho'|^2(s)ds=0.$$
Since $|\rho'|$ is nonincreasing, we have $s|\rho'|^2(s)\geq t|\rho'|^2(2t)$ for all $t\leq s\leq 2t$, so the last displayed equation finishes the proof.
\end{proof}
It is worth noting that passing to $T\rightarrow\infty$ in 
$$W_2(\rho(t,\cdot),\rho(T,\cdot))\leq \int_t^T|\rho'|(s)ds$$
yields 
$$W_2(\rho(t,\cdot),\rho_\infty)\leq \int_t^\infty|\rho'|(s)ds$$
so, by \eqref{asymp3}, we obtain
$$2\int_t^\infty (s-t)|\rho'|^2(s)ds\leq \bigg(\int_t^\infty|\rho'|(s)ds\bigg)^2.$$
However, this inequality gives no extra information on $|\rho'|$, since it is true that:
\begin{lemma}\label{general-omega}
For any nonincreasing, nonnegative function $\omega$ on $[0,\infty)$ which satisfies the necessary integrability conditions we have that 
$$2\int_t^\infty (s-t)\omega^2(s)ds\leq \bigg(\int_t^\infty\omega(s)ds\bigg)^2\mbox{ for all }t\geq0.$$
\end{lemma}
\begin{proof}
Let 
$$\phi(t):=2\int_t^\infty (s-t)\omega^2(s)ds-\bigg(\int_t^\infty\omega(s)ds\bigg)^2\mbox{ for all }t\geq0.$$
The weak derivative of $\phi$ as an absolutely continuous function is
\begin{eqnarray*}
\dot\phi(t)&=&-2t\omega^2(t)-2\int_t^\infty\omega^2(s)ds+2t\omega^2(t)+2\omega(t)\int_t^\infty\omega(s)ds\\
&=&2\int_t^\infty[\omega(t)-\omega(s)]\omega(s)ds\geq0
\end{eqnarray*}
since $\omega$ is nonnegative and nonincreasing. We end the proof by observing that $\phi(\infty)=0$.
\end{proof}
We learn more about the asymptotic behavior of $\rho$ (or, equivalently, of $N$) from:
\begin{prop}\label{spacetime}
Let $\rho_0 \in \mathcal{P}_2(\Reals)$ and  $v_0\in C_b(\Reals)$. Then the function
\begin{equation}\label{t-x}
P(t,x):=\int_0^x N(t,z)dz\mbox{ is concave in }t\geq0\mbox{ and convex in }x\in[0,1].
\end{equation}
\end{prop}
\begin{proof}
The convexity in $x$ is obvious due to the monotonicity of $N(t,\cdot)$. 
Now, once again, from \cite{Natile2009} we have that $P(t,\cdot)$ is the largest convex function such that
$$P(t,x)\leq\int_0^x\big[N_0(z)+tV_0(z)\big]dz\mbox{ for all }t\geq0\mbox{ and }x\in[0,1]$$
and 
$$P(t,0)=0,\ P(t,1)=\int_0^1\big[N_0(z)+tV_0(z)\big]dz\mbox{ for all }t\geq0.$$
Thus, if $0\leq s\leq t$, we have
$$\frac{1}{2}\big[P(s,x)+P(t,x)\big]\leq\int_0^x\bigg[N_0(z)+\frac{s+t}{2}V_0(z)\bigg]dz\mbox{ for all }x\in[0,1].$$
Note that
$$\frac{1}{2}\big[P(s,0)+P(t,0)\big]=0\mbox{ and }\frac{1}{2}\big[P(s,1)+P(t,1)\big]=\int_0^1\bigg[N_0(z)+\frac{s+t}{2}V_0(z)\bigg]dz$$
and the function 
$$[0,1]\ni x\mapsto \frac{1}{2}\big[P(s,x)+P(t,x)\big]\mbox{ is convex}.$$
But
$$P\bigg(\frac{s+t}{2},\cdot\bigg)\mbox{ is the largest convex function}$$
that satisfies the previous three displayed conditions, so we are done.
\end{proof}
In the confined case we know from \eqref{Nt-cv} that for all $x\in[0,1]$, $P(\cdot,x)$ has a finite asymptotic limit, so we can easily deduce:
\begin{corollary}\label{Nt-decrease}
Once again, assume $\rho\in \mathcal{P}([0,1])$ and $v_0\in C(\Reals)$ with $v_0(0)=v_0(1)=0$. If the SP evolution is  confined to $[0,1]$, then $[0,\infty)\ni t\mapsto P(t,x)$ is a nondecreasing, concave function, such that 
$$\lim_{t\rightarrow\infty}P(t,x)=\int_0^x N_\infty(z)dz\mbox{ uniformly for all } x\in[0,1].$$
\end{corollary}

The equilibria are unstable with respect to the initial data, as this very simple example shows. Let $\rho_0^n:=(\delta_{1/n}+\delta_{1-1/n})/2$ with $v_0^n(1/n)=1/n$ and $v_0^n(1-1/n)=-1/n$, which can be chosen arbitrarily close to the equilibrium solution $\rho_\infty=(\delta_0+\delta_1)/2$, $v\equiv0$. However, the equilibrium solution for any integer $n\geq2$ is $\rho^n_\infty=\delta_{1/2}$.

\noindent {\bf Rates of decay}

In view of \eqref{asymp1}, if we denote by $\mu$ the finite Borel measure on $(0,\infty)$ whose density is $|\rho'|^2$, we can apply the de la Vall\'{e}e-Poussin Lemma to \eqref{asymp2} to infer the existence of a strictly convex, superlinear, nonnegative, $C^\infty$ function $\Psi$ on $[0,\infty)$ such that $\Psi\in L^1(\mu)$ and $\Psi(0)=0$. 
\begin{prop}\label{power-law1}
Suppose there exist $a,\ \gamma>0$ such that $\Psi(t)\geq at^{1+\gamma}$ for all $t\geq0$. Then 
$$e(t)\leq C(\rho_0,v_0)t^{-\gamma}\mbox{ for all }t>0,$$
where 
$$C(\rho_0,v_0):=\frac{1}{a\gamma}\int_0^\infty \Psi(s)|\rho'|^{2}(s)ds.$$
\end{prop}
\begin{proof}
Let $t>0$. By Cauchy-Schwarz, we have
\begin{eqnarray*}
\bigg(\int_t^\infty|\rho'|(s)ds\bigg)^2&\leq& \int_t^\infty \frac{ds}{\Psi(s)}\,\int_t^\infty \Psi(s)|\rho'|^{2}(s)ds\\
&\leq& \frac{1}{a}\int_t^\infty s^{-1-\gamma}ds \,\int_0^\infty \Psi(s)|\rho'|^{2}(s)ds\\
&=& C(\rho_0,v_0) t^{-\gamma},
\end{eqnarray*}
which finishes the proof.
\end{proof}
Other type of conditions on $|\rho'|$ lead to power-law decay. For example:
\begin{prop}\label{power-law2}
Suppose there exists $1/2<p<1$ such that $g(t):=t^{1/2}|\rho'|(t)$ belongs to $L^{2p}(0,\infty)$. Then 
$$e(t)\leq C(\rho_0,v_0;p)t^{1-\frac{1}{p}}\mbox{ for all }t>0,$$
where 
$$C(\rho_0,v_0;p):=\bigg(\frac{2p-1}{1-p}\bigg)^{2-\frac{1}{p}}\|g\|^2_{L^{2p}(0,\infty)}.$$
\end{prop}
\begin{proof}
Let $t>0$. H\"{o}lder's inequality yields
$$\bigg(\int_t^\infty|\rho'|(s)ds\bigg)^2\leq \bigg(\int_t^\infty s^{-q/2}ds\bigg)^\frac{2}{q}\bigg(\int_t^\infty s^p|\rho'|^{2p}(s)ds\bigg)^\frac{1}{p}$$
for $p>1/2$ and $q>1$ such that 
$$\frac{1}{2p}+\frac{1}{q}=1.$$
So,
$$e(t)\leq C(\rho_0,v_0;p)t^{\frac{2}{q}-1}=C(\rho_0,v_0;p)t^{1-\frac{1}{p}}.$$
Note that $q>2$ is equivalent to $p<1$ and the conclusion follows.
\end{proof}
 \begin{figure}[h]
\scalebox{0.90}
 {\includegraphics[width=\textwidth]{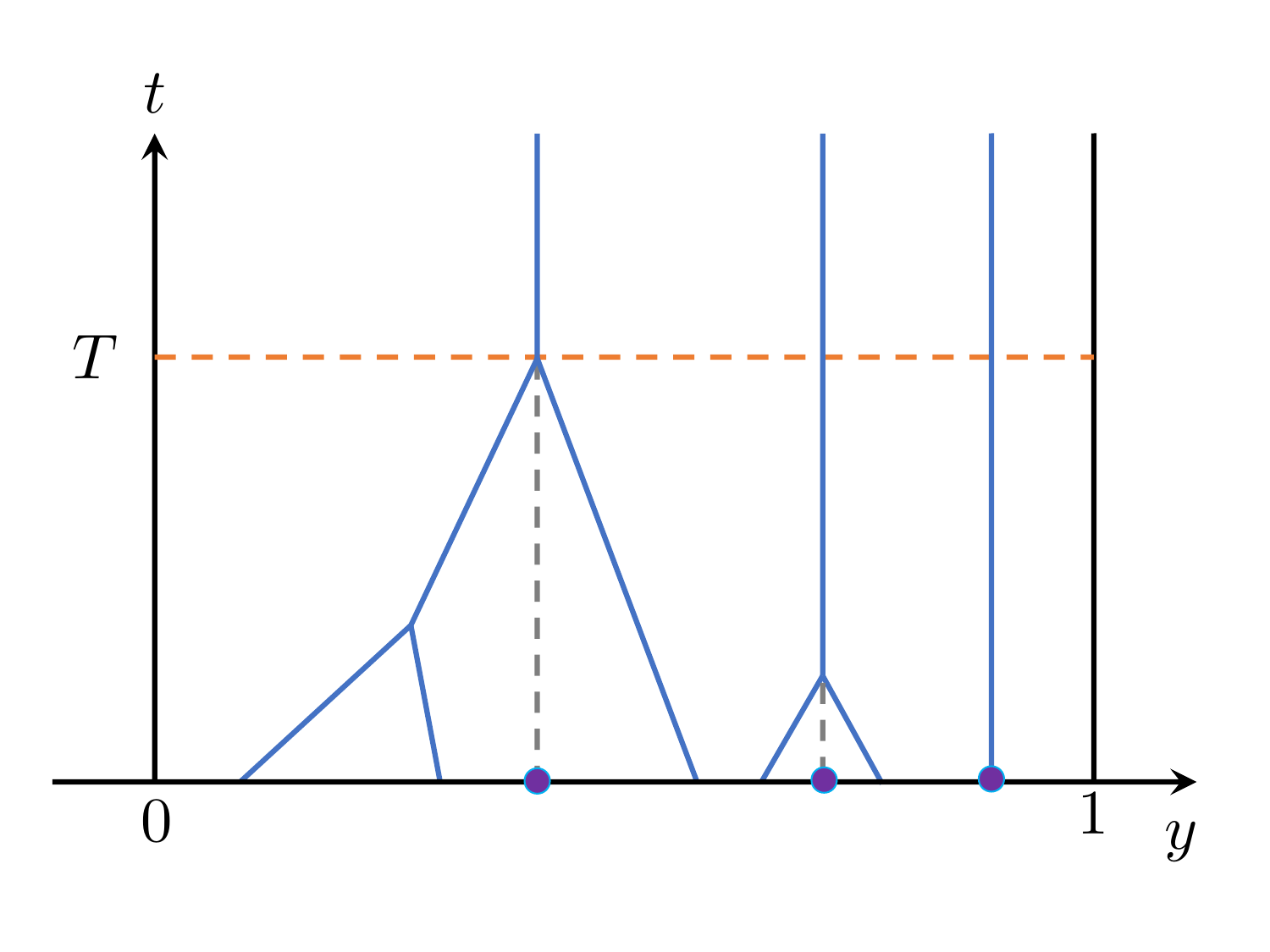}}
  \caption{{Confined SP evolution starting at a finite, convex combination of Dirac masses. The equilibrium is reached at $t=T$ and consists of a linear combination of Dirac masses situated at the three highlighted locations.}}
  \label{states4}
\end{figure}
\begin{remark}\label{finite-discrete}
The above two propositions apply to finite convex combinations of Dirac masses $\rho_0$ and velocities $v_0$ which guarantee that the SP evolution is confined within a compact interval. Indeed, in all these cases, the asymptotic limit is actually reached in finite time and the conditions on $|\rho'|$ from Proposition \ref{power-law1} and Proposition \ref{power-law2} are trivially satisfied for all $1/2<p<1$ and all $\gamma>0$.
\end{remark}
\section{Questions about the rate of decay}
Is it possible that $e$ can decay as slowly as possible, i.e. for any convex function $\omega$ on $[0,\infty)$ that decays asymptotically to zero there are initial conditions for which $e\geq \omega$ on $[T,\infty)$ for some $T>0$? Let us consider $\{x_k\}_{k\geq0}$, $\{M_k\}_{k\geq0}$ to be two increasing sequences of positive reals converging to 1. The masses $m_0:=M_0$ and $m_k:=M_k-M_{k-1}$ (chosen such that $\{m_k\}_{k\geq0}$ is decreasing) for $k\geq1$ are assigned to the locations $x_k$. Furthermore, let $\{b_k\}_{k\geq 1}$ be a decreasing sequence of positive numbers converging to 0. The particle of mass $m_0$ located at $x_0$ starts traveling with some velocity $a>0$, while from $x_k$, $k\geq1$ the mass $m_k$ travels with velocity $-b_k$. The points of contact are $(t_k,y_k)$ and the velocity of the trajectory starting at $x_0$ is denoted by $v_k$ between the times $t_k$ and $t_{k+1}$. After some straightforward albeit somewhat laborious calculations we obtain the following recursive formulae
\begin{equation}\label{vk}
v_k=\frac{M_{k-1}v_{k-1}-m_kb_k}{M_k},
\end{equation}
\begin{equation}\label{tk}
t_k=\frac{x_k-y_{k-1}+t_{k-1}v_{k-1}}{v_{k-1}+b_k},
\end{equation}
and 
\begin{equation}\label{yk}
y_k=\frac{x_kv_{k-1}+b_ky_{k-1}-t_{k-1}v_{k-1}b_k}{v_{k-1}+b_k}.
\end{equation}
 \begin{figure}[h]
\scalebox{0.90}
 {\includegraphics[width=\textwidth]{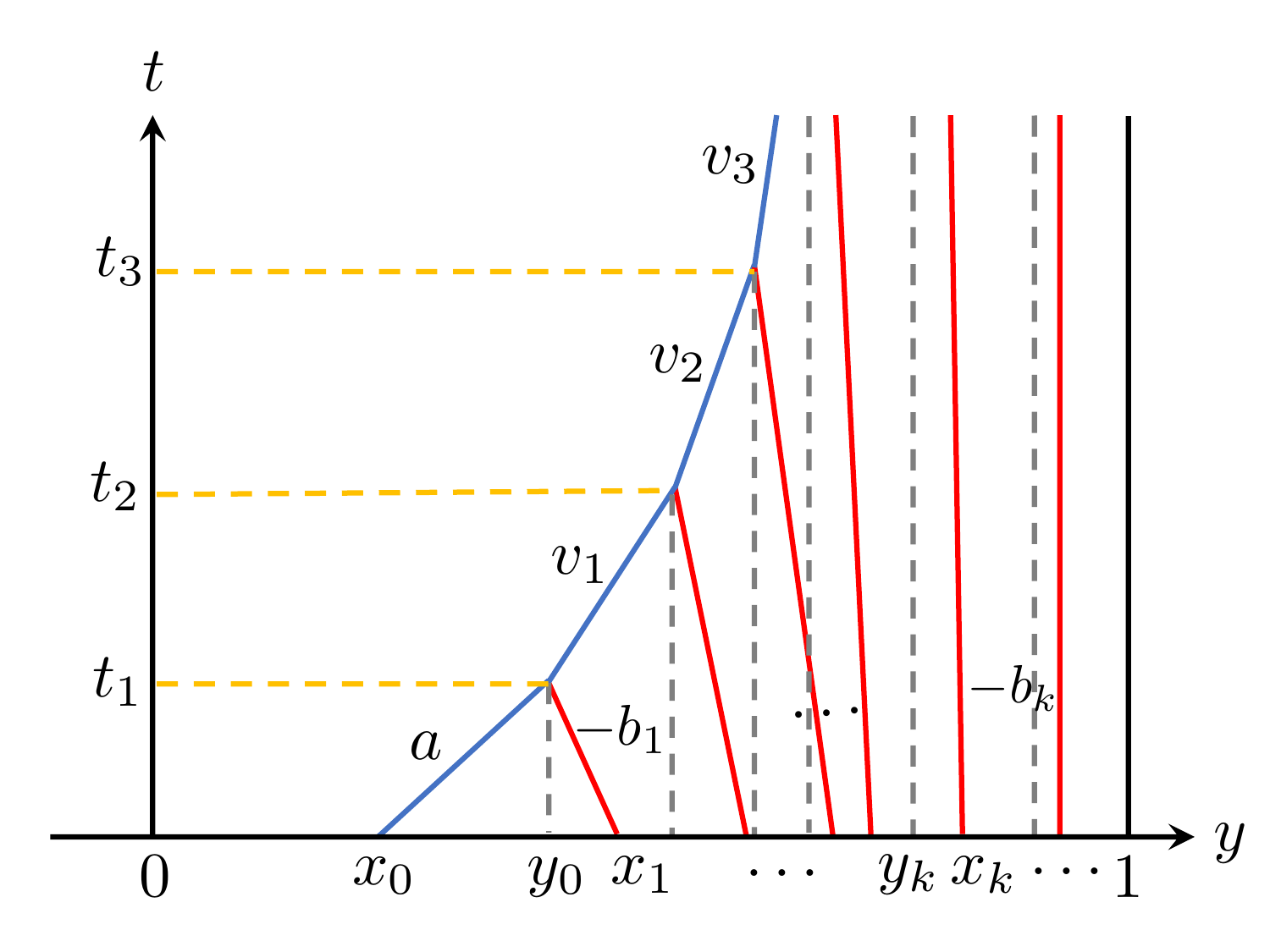}}
  \caption{{The mass starting at $x_0$ initially traveling to the right with velocity $a$ is ``bombarded'' successively by the masses starting at $x_k$, $k\geq1$ and traveling to the left with velocities $-b_k$.}}
  \label{states5}
\end{figure}
Here $t_0=0$, $v_0=a$, $y_0=x_0$. Note that $a$ must be chosen such that $v_k\rightarrow0^+$ as $k\rightarrow\infty$. By \eqref{vk} we get
$$M_kv_k=am_0-\sum_{i=1}^km_ib_i\mbox{ for all }k\geq1,$$
which, in light of $M_k\rightarrow 1^-$ as $k\rightarrow\infty$, reveals 
$$\lim_{k\rightarrow\infty}v_k=0\mbox{ if and only if }a=\frac{1}{m_0}\sum_{i=1}^\infty m_ib_i.$$
With the freedom to choose the sequences $\{x_k\}_{k\geq0}$, $\{M_k\}_{k\geq0}$ and $\{b_k\}_{k\geq1}$ with the properties described above, it will be interesting to calculate the rate of decay for $e(t)$; note that $\rho(t,\cdot)$ can be explicitly calculated (it is a countable convex combination of point masses) and, since the equilibrium solution is $\rho_\infty=\delta_{\bar y}$, (where $\bar y$ is the limit of the increasing sequence $\{y_k\}_k$) we get an explicit formula for $e(t)$. 
We have 
$$\rho(t_k,\cdot)=M_k\delta_{y_k}+\sum_{j=k+1}^\infty m_j\delta_{x_j-t_kb_j},$$
so that
$$e(t_k)=W_2^2(\rho(t_k,\cdot),\delta_{\bar y})=M_k(\bar y-y_k)^2+\sum_{j=k+1}^\infty m_j(\bar y-x_j+t_kb_j)^2.$$

\noindent {\bf Questions 1, 2:} 1. For different countable infinite combinations and corresponding initial velocities, are there exponents $\gamma>0$ such that $t_k^{\gamma}e(t_k)$ is bounded? 2. Furthermore, is there such a $\gamma$ which works for all initial configurations of the type shown in Figure \ref{states5}?

\noindent The initial configuration $\rho_0$ is a countable infinite convex combination of Dirac masses; this is an essential point to be made, because for finite convex combinations the equilibria are arrived at in finite time. By taking $x_k=1-0.5^{k+1}$, $m_k=0.5^{k+1}$ , $b_{k+1}=0.5^{k+1}$ for $k\geq0$, we have been able to construct explicit solutions (using Mathematica) with decay rate of exactly $t^{-1}$ for $e(t)$. Using Matlab, we tried different values for the above sequences and numerical evidence seems to indicate that by taking $b_k:=n^{-k}$, the rate of decay slows down and it is of the form $t^{-\gamma}$ with $\gamma\rightarrow0^+$ as $n\rightarrow\infty$. So the answer to Q2 seems to be negative, which means that, likely, a general power-law decay is unavailable.

\noindent {\bf Question 3.} Since the answer to Q1 appears to be negative, is it possible that the decay to equilibrium can be as slow as possible? In other words, let  $\omega$ be a convex function on $[0,\infty)$ that decays asymptotically to zero.  Are there sequences $\{x_k\}_{k\geq0}$, $\{M_k\}_{k\geq0}$ and $\{b_k\}_{k\geq1}$ as above for which $e(t_k)\geq \omega(t_k)$ for sufficiently large $k$? 

\noindent {\bf Q4.} Are there different types of initial data that can help answer Q3?

\section{An alternate proof for the existence of the asymptotic limit}
Here we show how the Lagrangian characterization of the SPS from \cite{Hynd2019} can be used to prove existence of the asymptotic limit without recourse to the projection characterization of the optimal maps in \cite{Natile2009}. 
Let us consider the equation 
$$
\dot X(t,\cdot)=\E_{\rho_0}[v_0|X(t,\cdot)]\;\; \text{a.e.}\; t\ge 0
$$
along with the initial condition
$$
X(0,\cdot)=\text{id}_\Reals.
$$
 In \cite{Hynd2019}, it is shown that if $\rho_0\in {\calP}_2(\R)$ and if the initial velocity $v_0: \Reals\rightarrow \Reals$ is absolutely continuous, then 
there is a Lipschitz continuous solution $X: [0,\infty)\rightarrow L^2(\rho_0)$ of this initial value problem. Moreover, the $X$ constructed was shown to have various properties including 
$$
\text{$y\mapsto X(t,y)$ is nondecreasing}\mbox{  on }\mathrm{spt}(\rho_0)\mbox{ for each }t\ge 0.$$

 \par Another property which was established is that for each pair of times $s,t$ with
  $0< s\le t$, there is a Lipschitz $f_{t,s}:\Reals\rightarrow \Reals$ such that 
  $$
  X(t,\cdot)=f_{t,s}(X(s,\cdot)).
  $$
It then follows that 
 \be\label{XteeIdentity}
  \int_{\Reals}X(t,y)\dot X(s,y)\rho_0(dy)=  \int_{\Reals}X(t,y)v_0(y)\rho_0(dy)
  \ee
  for almost every $s\in [0,t]$.  This observation leads to the subsequent assertion. 
\begin{lemma}
For each $0\le s\le t$, 
\be\label{projectionIdentity}
  \int_{\Reals}X(t,y) X(s,y)\rho_0(dy)=  \int_{\Reals}X(t,y)[y+s v_0(y)]\rho_0(dy).
\ee
\end{lemma}  
\begin{proof}
Set 
$$
g(s):=  \int_{\Reals}X(t,y) \big\{X(s,y)-[y+s v_0(y)]\big\}\rho_0(dy)
$$
for $s\in [0,t]$. Note that
$$
g'(s)=  \int_{\Reals}X(t,y) [\dot X(s,y)- v_0(y)]\rho_0(dy)=0
$$
for almost every $s\in [0,t]$ by \eqref{XteeIdentity}. Thus, $g(s)=g(0)=0$ for 
all $s\in [0,t]$. 
\end{proof}

\par The situation we have been interested in is when the flow is bounded.
\begin{theorem}
Suppose there is a real number $R$ such that 
$$
|X(t,y)|\le R\; \text{ for all $y\in \textup{spt}(\rho_0)$ and $t\ge 0$.}
$$
Then the limit 
$$
\lim_{t\rightarrow\infty}X(t,\cdot)\mbox{ exists in }L^2(\rho_0).$$
\end{theorem}

\begin{proof}
Let $\{t_n\}_{n\in \N}$ be a sequence of nonnegative numbers increasing to infinity. Then 
$$
X(t_n,\cdot):\textup{spt}(\rho_0)\rightarrow \Reals
$$
defines a bounded sequence of monotone functions indexed by $n\in \N$.  By Helly's selection theorem, there is a subsequence (not relabeled) and a monotone function $Y$ such that 
$$
X(t_n,y)\rightarrow Y(y)
\mbox{ for all }y\in\textup{spt}(\rho_0).$$  It also follows easily that $X(t_n,\cdot)\rightarrow Y$ as $n\rightarrow\infty$ in $L^2(\rho_0)$. 

\par Letting $t=t_n$ in the identity \eqref{projectionIdentity} and sending $n\rightarrow\infty$ gives  
$$
  \int_{\Reals}Y(y) X(s,y)\rho_0(dy)=  \int_{\Reals}Y(y)[y+s v_0(y)]\rho_0(dy)
$$
for all $s\ge 0$. Since $Y$ and $X(s,\cdot)$ assume all their values in $[-R,R]$, it must be that 
$$
 \int_{\Reals}Y(y) v_0(y)\rho_0(dy)=0.
$$
Therefore, 
\be\label{YXsIdentity}
  \int_{\Reals}Y(y) X(s,y)\rho_0(dy)=  \int_{\Reals}yY(y)\rho_0(dy)
\ee
for all $s\ge 0$.  And sending $s=t_n\rightarrow\infty$ gives 
\be\nonumber
  \int_{\Reals}Y^2(y)\rho_0(dy)=  \int_{\Reals}yY(y)\rho_0(dy).
\ee

\par Let us now assume there is another sequence $\{s_n\}_{n}\rightarrow\infty$ such that $X(s_n,\cdot)$ converges pointwise to a limit function $Z$. Then the arguments above give 
$$
 \int_{\Reals}Z(y) v_0(y)\rho_0(dy)=0
\quad \text{and}\quad  \int_{\Reals}Z^2(y)\rho_0(dy)=  \int_{\Reals}yZ(y)\rho_0(dy).
$$
We may also send $s=s_n\rightarrow\infty$ in \eqref{YXsIdentity} to get 
$$
  \int_{\Reals}Y(y) Z(y)\rho_0(dy)=  \int_{\Reals}yY(y)\rho_0(dy).
$$
Likewise, we find
$$
  \int_{\Reals}Z(y)Y(y)\rho_0(dy)=  \int_{\Reals}yZ(y)\rho_0(dy).
$$

\par Employing the various equalities we have derived involving $Y$ and $Z$ gives 
\begin{align}\nonumber
\|Y-Z\|^2_{L^2(\rho_0)}&= \|Y\|^2_{L^2(\rho_0)}+ \|Z\|^2_{L^2(\rho_0)}-2\langle Y,Z\rangle_{L^2(\rho_0)}\\
&= \langle\text{id}_\Reals,Y\rangle_{L^2(\rho_0)}+ \langle\text{id}_\Reals,Z\rangle_{L^2(\rho_0)}-2 \langle Y,Z\rangle_{L^2(\rho_0)}\nonumber\\
&=0\nonumber.
\end{align}
Thus, $Y(y)=Z(y)$ for $\rho_0$ almost every $y\in \R$. Since the sequential limits of the family $\{X(t)\}_{t\ge 0}$ in $L^2(\rho_0)$ are independent of the sequence, the proof is concluded.
\end{proof}
Let us denote 
$$
X^\infty:=\lim_{t\rightarrow\infty}X(t,\cdot).
$$
 A direct corollary of the theorem is that the particle distribution $X(t)_{\#}\rho_0$ converges to $X^\infty_{\#}\rho_0$ in ${\calP}_2(\Reals)$ as $t\rightarrow\infty$. Adapting the proof of the theorem, we may also conclude that 
$$
\int_{\R}h(X^\infty(y))v_0(y)\rho_0(dy)=0
$$
for any bounded Borel $h: \Reals\rightarrow \Reals$.  That is, $\E_{\rho_0}[v_0|X^\infty]=0$.

\par Finally, we note that if we only required 
$$
\sup_{t\geq0}\|X(t,\cdot)\|_{L^2(\rho_0)}<\infty,
$$
we would be able to conclude that the limit  $\displaystyle\lim_{t\rightarrow\infty}X(t,\cdot)$ exists weakly in $L^2(\rho_0)$. 

\section*{Acknowledgements}
Ryan Hynd was partially supported by an AMS Claytor-Gilmer fellowship.

\end{document}